\documentclass{sig-alt-full}

\usepackage[utf8]{inputenc}
\usepackage[T1]{fontenc}
\usepackage{lmodern}

\usepackage{comment}
\usepackage{amssymb,amsmath}
\usepackage{theorem}
\usepackage{graphics}
\usepackage{amsfonts}
\usepackage{mathrsfs}
\usepackage{amscd}
\usepackage{color}
\usepackage{url}
\usepackage{alltt}
\usepackage{bbm}

\usepackage{microtype}  
\usepackage{mathtools}  
\usepackage{booktabs}   
\usepackage{array}      

\newdef{defi}{Definition}
\newtheorem{thm}{Theorem}
\newtheorem{prop}{Proposition}

\begin{document}
\crdata{}

\title{Elementary Integration of Superelliptic Integrals}
\newfont{\authfntsmall}{phvr at 11pt}
\newfont{\eaddfntsmall}{phvr at 9pt}

\numberofauthors{3} 
\author{
\alignauthor
Thierry Combot\\
       \affaddr{Univ. de Bourgogne (France)}\\
       \email{$\!\!\!\!\!\!\!\!\!\!\!\!$thierry.combot@u-bourgogne.fr$\!\!\!\!\!\!\!\!\!\!\!\!$}
}
\date{today}

\maketitle

\begin{abstract} 
Consider a superelliptic integral $I=\int P/(Q S^{1/k}) dx$ with $\mathbb{K}=\mathbb{Q}(\xi)$, $\xi$ a primitive $k$th root of unity, $P,Q,S\in\mathbb{K}[x]$ and $S$ has simple roots and degree coprime with $k$. Note $d$ the maximum of the degree of $P,Q,S$, $h$ the logarithmic height of the coefficients and $g$ the genus of $y^k-S(x)$. We present an algorithm which solves the elementary integration problem of $I$ generically in $O((kd)^{\omega+2g+1} h^{g+1})$ operations.
\end{abstract}

\bigskip
\vspace{1mm}
\noindent
{\bf Categories and Subject Descriptors: 68W30} \\

\vspace{1mm}

A superelliptic integral is an integral of the form
$$I(x)=\int \frac{P(x)}{Q(x) S(x)^{1/k}} dx$$
with $P,Q,S\in\mathbb{K}[x]$, and we can assume the multiplicity of roots of $S$ to be $<k$. We add \emph{the technical condition} that the roots of $S$ are simple and its degree is coprime with $k$. If the degree is multiple of $k$, then infinity is a regular point, and a root of $S$ can be sent to infinity by a Moebius transformation, which however often requires a field extension of the coefficients. The case $k=1$ is well known \cite{1}, is always elementary integrable, and thus we can assume $k\geq 2$.  When $k=2$, the roots of $S$ are always simple, and the integral $I$ is hyperelliptic. For general $k$, the integral $I$ is called superelliptic, and is, although of specific form, wildly encountered, see for example \cite{2,8} where all ``random'' examples where in fact superelliptic. The purpose of this article is to present an efficient algorithm to decide if $I$ is an elementary integral.\\

\begin{defi}
A superelliptic integral is called
\begin{itemize}
\item elementary if $I$ can be written
$$I(x)=G_0(x)+\sum\limits_{i} \lambda_i \ln G_i(x), \; G_i\in \overline{\mathbb{Q}}(x,S^{1/k})$$\vspace{-0.6cm}
\item reduced if $Q$ is square free coprime with $S$ and $\deg P <\deg Q+k^{-1} \deg S-1$
\item of first kind if reduced and $Q$ is constant.
\item of torsion if it is a sum of an elementary integral and a first kind integral.
\end{itemize}
\end{defi}

The integration process starts with a Hermite reduction, which consists in finding an algebraic function $G_0$ such that $\partial_x G_0-PQ^{-1}S^{-1/k}$ has only simple poles and no pole at infinity. If such function can be found, the resulting integral is reduced. This part is well known \cite{2,3}, and in our case is even simpler as the notion of integral basis can be avoided. This part and its complexity will be recalled at the beginning of section $2$.\\

Much more difficult is to find the logarithmic part. An integral of the first kind is never elementary except $0$. If we are able to find a sum $L$ of logarithms of algebraic functions such that their residues coincides with the ones of the integral, the difference $I-L$ is of first kind and thus $I$ is of torsion, and then $I$ is elementary if and only if $I-L=0$. This is where expedients as heuristic approaches are used to speed up the computation \cite{8}. The base approach is given by Trager \cite{1}, and improved by Bertrand in the hyperelliptic case. Three major difficulties happen in these approaches.\\

\textbf{Problem 1}. For a reduced superelliptic integral, the re\-si\-dues $\lambda$ are roots of the polynomial
$$R(\lambda)=\hbox{resultant}_x(P^k-\lambda^k Q'^k S,Q),$$
which factorizes in $\mathbb{K}[\lambda]$ under the form
$$R(\lambda)=\prod_{i=1}^l R_i(\lambda^{k_i}),\quad k_i \mid k,$$
We say that $R$ is \emph{generic} if for all $i$ the Galois group of $R_i(\lambda^{k_i})$ is maximal $\mathbb{Z}_{k_i}^{\deg R_i} \ltimes S_{\deg R_i}$. Following Trager, we need to compute a $\mathbb{Q}$-basis of these residues. However, generically $l=1$ and $R$ is generic, and thus the splitting field of $R$ is of degree $k^dd!$. Worse, this computation has to be done at the beginning, even if $I$ was not elementary in the end. The generic case is thus typically intractable \cite{5}, which is probably the main reason this algorithm is still not implemented in Maple and Mathematica. Theorem \ref{thm1} presents a similar decomposition much cheaper, and sufficient to conclude in the generic case.

\begin{thm}\label{thm1}
Given a field extension $\mathbb{K}[\alpha]$ and a reduced superelliptic integral $I$, the algorithm \underline{\sf TraceIntegrals} computes a set of integrals $\mathcal{S}_\alpha$ such that
\begin{itemize}
\item The residues of the integrals are in $\mathbb{Z}[\xi]$.
\item If the integral $I$ is elementary, then all integrals are of torsion.
\end{itemize}
If $R$ is generic, $I$ is a linear combination of the $(\mathcal{S}_{\alpha})_{R(\alpha)=0}$ and an integral of the first kind.
\end{thm}

The generic condition is sufficient but not necessary, and no examples were found for which $I$ could not be decomposed as such. Taking a larger $\mathbb{K}[\alpha]$ containing simultaneously all the roots of $R$ reduces to original Trager's approach and thus provably works in all cases, but the cost then rises to $O((k^dd!)^3)$. Remark however that even if the decomposition is not possible, the integrals $\mathcal{S}_\alpha$ being of torsion is still a necessary condition for elementary integration, and thus gives a quick test to prove an integral is non-elementary.\\

A particular behaviour of superelliptic integrals is the Galois action on the $k$-th root, which multiplies the integral by $\xi$. Because of this, the space of residues of a superelliptic integral is always invariant by $\xi$. Thus our decomposition is similar to Trager's one but done on $\mathbb{K}$ instead of $\mathbb{Q}$, as keeping this invariance by multiplication by $\xi$ is essential for the following.\\

\textbf{Problem 2}. For each integral of $\mathcal{S}_\alpha$, we need to decide if they are of torsion. Possibly more than one log is necessary, but they have to respect a precise pattern to ensure the invariance by $\xi$, leading us to introduce
$$L_S(P)= \sum\limits_{i=0}^{k-1} \xi^i\ln\left( \sum\limits_{j=0}^{k-1} P\left(x,\xi^j S^{1/k}\right) \right)$$
with $P\in\mathbb{C}[y]_{\leq k-1}[x]$. Now, a single function $L_S$ is necessary (see Propositions \ref{prop1},\ref{prop1b}). Trager's approach is to build a divisor on the superelliptic curve, and then to test its principality using linear algebra. Bertrand uses a nice representation of divisors on hyperelliptic curves to make the process more efficient. However, both are polynomial in the size of the output, and the degree of $P$ depends on the residues, and so can be extremely large as they are not even bounded by a function of $d$. It happens however that testing principality of a divisor can be done in logarithmic time of its height.

\begin{thm}\label{thm2}
Consider a reduced superelliptic integral $I$ with residues in $\mathbb{Z}[\xi]$. If $I= L_S(P)$ up to an integral of the first kind for a polynomial $P$, then it can be written
$$I= \sum\limits_{i=0}^l (-2)^i L_S(P_i)$$
up to a first kind integral where $\deg_x P_i \leq \deg Q +\frac{1}{k}\deg S$ and $l\leq \underset{r \hbox{ residues}}{\max} (\hbox{log}_2 \mid r \mid ) $. Algorithm \underline{\sf JacobianReduce} compute this decomposition in time $O((kd)^\omega l)$ if it exists.
\end{thm}

Remark that simplifying the sum to have just one $L_S$ function would have an exponential cost in $l$, thus this is thank to this specific representation of the solution that such a fast algorithm is possible. Testing principality of a divisor is equivalent to test if its reduction in the Jacobian of the superelliptic curve is $0$. The technical condition allows to have an efficient representation of divisors. If the support of the divisor is small but height is large, then a fast multiplication technique in the Jacobian is very efficient \cite{4}. This was overlooked by Bertrand \cite{3}. This allows a fantastic speed up as Trager and Bertrand algorithm were exponential in logarithmic height, and ours is linear. However, the coefficients size grows fast which decreases the usefulness of the algorithm, except when computed modulo a prime number. This gives a quick test for proving that the divisor is not principal, which is generically the case as, except in genus $0$, most superelliptic integrals are not elementary.\\
 
\textbf{Problem 3}. Theorem \ref{thm2} does not solve the elementary integration problem, as we could have $I=\frac{1}{N} L_S(P)$ with $N\in\mathbb{N},\; N\geq 2$ and this would not be found by Theorem \ref{thm2}. In this case, the corresponding divisors are not principal but of torsion, i.e. a multiple of the divisor is principal. This problem is solved by Trager using two ``good reduction primes'' $p$ \cite{1}, and then deducing a unique possible candidate for $N$. We know that all divisors are of torsion modulo $p$. Thus given a divisor and two good reductions $p,q$, we test the principality of multiples of the divisor modulo $p$ and $q$ until we find a multiple principal. A large prime is used to confirm with high probability this candidate.

\begin{thm}\label{thm3}
Consider a reduced superelliptic integral $I$ with residues in $\mathbb{Z}[\xi]$ with coefficients in $\mathbb{K}(\alpha)$ with $\alpha$ algebraic of degree $r$. Note $\Delta$ the discriminant of the square free part of $QS$. If $I=\frac{1}{N} L_S(P)$ for a polynomial $P$ and $N\in\mathbb{N}^*$, then algorithm \underline{\sf TorsionOrder} finds a non zero multiple of $N$ which is expected to be less than $\;\;\;$ $(1+\sqrt{r\phi(k)\ln(k\Delta)})^{2g}$
in time $\tilde{O}( (kd)^{\omega+g} h^{g+1} r^g)$. Else algorithm \underline{\sf TorsionOrder} returns $0$ with probability $1-\epsilon$.
\end{thm}

The probabilistic part is important because having a false positive can be very costly in characteristic zero, as the binary complexity of \underline{\sf JacobianReduce} is not logarithmic in divisor height.

\section{Integral Decomposition}

\subsection{Hermite reduction}

The Hermite reduction for algebraic integral is described in \cite{2,3,8}. In our simpler superelliptic case, let us recall it to precise its complexity. We note $\hbox{sf}(Q)$ square free part of $Q$.\\

\noindent\underline{\sf HermiteReduction}\\
\textsf{Input:} A superelliptic integral $P/(QS^{1/k})$.\\
\textsf{Output:} A rational function $G\in\mathbb{K}(x)$ such that the integral
$$\int \frac{P}{Q S^{1/k}}-\partial_x \left( \frac{G}{S^{1/k}}\right) dx$$
is reduced or ``FAIL''.
\begin{enumerate}
\item Note $\tilde{Q}=Q/\hbox{sf}(Q),\; \hat{Q}=\hbox{sf}(Q)/(\hbox{sf}(Q) \wedge S)$ and $G=T/\tilde{Q}$ with $T$ an unknown polynomial with $\deg T\leq \deg P -\deg \hbox{sf}(Q) +1$
\item Solve the linear system coming from the condition
$$ \!\!\!\!\!\! \!\!\!\!\!\!\left(\frac{P}{QS^{1/k}}-\partial_x \left(\frac{T}{\tilde{Q} S^{1/k}}\right)\right) S^{1/k}\hat{Q}  \in \mathbb{K}_{< \deg \hat{Q}+\frac{1}{k} \deg S-1} [x]$$
\item If one solution return $T/\tilde{Q}$ else ``FAIL''
\end{enumerate}

Such reduction is not always possible due to the condition on the degree at infinity, as for example elliptic integral of the second kind are not reducible by this process.

\begin{prop}
If $\int  \frac{P}{Q S^{1/k}}$ is elementary, the algorithm $\;$ \underline{\sf HermiteReduction} is successful and runs in $O(d^\omega)$.
\end{prop}

\begin{proof}
If $\int  \frac{P}{Q S^{1/k}}$ is elementary, then it can be written as a sum of an algebraic function and logs of algebraic functions. Thus $\frac{P}{Q S^{1/k}}$ can be written as the derivative of an algebraic function plus an algebraic function with simple poles. Now if an extension of $\mathbb{C}(x,S^{1/k})$ was necessary, then acting the Galois group on this extension would allow to find another decomposition in $\mathbb{C}(x,S^{1/k})$. Thus we can write
$$\frac{P}{Q S^{1/k}}=\partial_x \left( \frac{T}{\tilde{Q}S^{1/k}}\right)+ \frac{R}{\hat{Q} S^{1/k}}$$
with $T,R\in\mathbb{C}[x]$.
Now multiplying both sides by $\hat{Q}S^{1/k}$, we obtain that $(\frac{P}{QS^{1/k}}-\partial_x \frac{T}{\tilde{Q} S^{1/k}})S^{1/k} \hat{Q}$ should be a polynomial. Knowing that $\frac{R}{\hat{Q} S^{1/k}}$ is the derivative of logs and as $k \wedge \deg S =1$, there are no residues at infinity and so the exponent at infinity is $<-1$. Thus we have $\deg R < \deg \hat{Q}+\frac{1}{k} \deg S -1$ which is exactly the condition of step $2$. Thus a solution $T$ as step $2$ should be found, and step $3$ then returns $G=T/\tilde{Q}$.
The complexity comes from step $2$ where a system of size $\deg(PS\hat{Q})$ should be solved. This size is linear in $d$, and thus the complexity is $O(d^\omega)$
\end{proof}

\subsection{Trace integrals}

Following Trager, if $\int P/(QS^{1/k}) dx$ is reduced and elementary, then it can be written
\begin{equation}
\int P/(QS^{1/k}) dx= \sum\limits_{i=1}^l \lambda_i \ln G_i(x)
\end{equation}
where $G_i\in\overline{\mathbb{K}}[x,S^{1/k}]$ and the $\lambda_i$ form a $\mathbb{Q}$ basis of the residues of $P/(QS^{1/k})$.

\begin{defi}
Consider field extension $\mathbb{L} \supset \mathbb{K}(\alpha) \supset \mathbb{K}$. The trace of $\beta \in\mathbb{L}$ over $\mathbb{K}(\alpha)$ is 
$$\hbox{tr}_{\mathbb{K}(\alpha)}(\beta)=-\frac{\hbox{coeff}_{z^{\deg T-1}}(T)}{\deg T}$$
where $T\in \mathbb{K}(\alpha)[z]$ is minimal unitary polynomial of $\beta$.
\end{defi}

\begin{prop}
We can build functions $P_i/(Q_i S^{1/k})$ without pole at $\infty$ such that $Q_i \mid Q$ and $\forall \beta \in Q^{-1}(0)$ we have
$$\hbox{res}_\beta \frac{P_i}{Q_i S^{1/k}}= d_i \hbox{coeff}_{\alpha^i} \hbox{tr}_{\mathbb{K}(\alpha)} \left(\hbox{res}_\beta \frac{P}{Q S^{1/k}}\right)$$
with $d_i\in\mathbb{N}^*$ chosen minimal such that all residues are in $\mathbb{Z}[\xi]$. Algorithm \underline{\sf TraceIntegrals} computes these functions in $O(d^2e^2)$ with $e=[\mathbb{K}(\alpha):\mathbb{K}]$.
\end{prop}

\noindent\underline{\sf TraceIntegrals}\\
\textsf{Input:} A reduced superelliptic integral $\int P/(QS^{1/k}) dx$ and $\alpha\in\overline{\mathbb{K}}$ with minimal polynomial $E\in\mathbb{K}[z]$.\\
\textsf{Output:} A list $P_j/(Q_jS^{1/k})$ with $P_j,Q_j\in\mathbb{K}(\alpha)[x]$ such that
\begin{equation}\label{eq3b}
\hbox{res}_\beta P_jQ_j^{-1}S^{-1/k}= d_j \hbox{coeff}_{\alpha^j} \hbox{tr}_{\mathbb{K}[\alpha]} \hbox{res}_\beta P_jQ_j^{-1}S^{-1/k} \in\mathbb{Z}[\xi],
\end{equation}
$\forall \beta\in Q^{-1}(0)$, $d_j\in\mathbb{N}^*$ minimal.
\begin{enumerate}
\item Factorize $Q=Q_1\dots Q_l$ in $\mathbb{K}(\alpha)$
\item Pose $d=1$. For $i=1\dots l$ do
\item If $y^k-S$ solves in $\mathbb{K}(\alpha)[x]/(Q_i)$, note $y\in\mathbb{K}(\alpha)[x]/(Q_i)$ one of its solutions, else go to next $i$.
\item Compute $t_j=\hbox{coeff}_{\alpha^j}\hbox{tr}_{\mathbb{K}(\alpha)} P(x)/(Q'(x)y),\; j=0\dots$ $\deg E-1$, and reassign $d_j$ the minimal multiple of $d_j$ such that $d_j t_j\in \mathbb{Z}[\xi]$.
\item Solve equation $R_{i,j}(x)-t_jQ_i'(x)y=0 \hbox{ mod } Q_i$ with $\deg R_{i,j} \leq \deg Q_i-1$. 
\item Return
$$\left\lbrace \left(d_j\sum\limits_{i=1}^l \frac{R_{i,j}}{Q_iS^{1/k}} \right)_{j=0 \dots \deg E-1}  \right\rbrace$$
\end{enumerate}

\begin{proof}
Let us first check that algorithm \underline{\sf TraceIntegrals} compute the integrals. Consider a factor $Q_i$ obtained in step $2$ and $\beta$ one of its roots. Either $y^k-S$ is irreducible or it fully factorizes as all its solutions in $y$ are the same up to a power of $\xi$. If it is irreducible, the residue $P(\beta)/(Q'(\beta)S(\beta)^{1/k})$ has a minimal polynomial in $\mathbb{K}(\alpha)[\lambda^k]$, and thus its second leading coefficient is $0$, and so the trace is $0$. If it factorizes, then step $5$ builds functions $R_{i,j}/(Q_iS)$ whose residues are the coefficients $t_j$ in $\alpha$ of the trace. As $Q_i$ is irreducible, it has $\deg Q_i$ simple roots on which neither $Q_i'$ or $S$ vanishes. Thus equation $R_{i,j}(x)-tQ_i'(x)y=0 \hbox{ mod } Q_i$ is an interpolation problem and thus admits a unique solution with $\deg R_i \leq \deg Q_i-1$. The integer $d_j$ is the minimal one such that $d_jt_j\in \mathbb{Z}[\xi]$ and should be a multiple of the old $d_j$ to still satisfy this same condition for previous $i$'s. In step $6$, the sum is made over all factors $Q_i$, and as they have distinct roots, equation \eqref{eq3b} is satisfied.
We factorize a polynomial of degree $d$ in $\mathbb{K}(\alpha)$, which costs $\tilde{O}(de)$. Step $3$ uses factorization in an extension of degree $de$, so $\tilde{O}(d^2e)$. Step $4$ uses a resultant to compute the minimal polynomial, which costs $O(d^2e)$. Step $5$ computes $e$ interpolations which costs $O(e^2d^2)$. Thus the global cost is $O(e^2d^2)$.
\end{proof}

\begin{prop}\label{proptor}
If $\int P/(QS^{1/k}) dx$ is reduced and of torsion, then the output of \underline{\sf TraceIntegrals} are integrals of torsion.
\end{prop}

\begin{proof}
Let us note $\beta_1,\dots,\beta_p$ the residues of the integral at singular points on the superelliptic curve $\mathcal{C}=\{(x,y)\in \mathbb{C}^2,y^k-S(x)\}$. There exists $M\in M_{p,l}(\mathbb{Q})$ such that $\beta= M \lambda$. The $\lambda$'s, $\beta$'s and the $\alpha$ are in some field extension $\mathbb{L}\supset \mathbb{K}$. Let us note $\tau_j$ the operator extracting the $\alpha^j$ coefficient of the trace over $\mathbb{K}(\alpha)$. As the trace $\mathbb{Q}$ linear, $\tau_j$ is also, and we have
$$\tau_j(\beta)=M \tau_j(\lambda).$$
As $\int P/(QS^{1/k}) dx= \sum\limits_{i=1}^l \lambda_i \ln G_i(x)$ up to an integral of first kind, each column of $M$ defines the list of residues of $\partial_x \ln G_i$ (and so are in fact integers). Note $\pi_1,\dots,\pi_{\phi(k)}$ the projectors to a basis $B$ of $\mathbb{Q}(\alpha)(\xi)$ over $\mathbb{Q}(\alpha)$ and the functions
$$F_s= \prod\limits_{i=1}^l G_i^{\tilde{d}_j\pi_s(\tau_j(\lambda_i))}$$
with $\tilde{d}_j\in\mathbb{N}^*$ such that all exponents are in $\mathbb{Z}$. The $\partial_x \ln F_s$ have for residues $\pi_s(\tau_j(\beta))$, and then $\partial_x \sum B_s \ln F_s$ has for residue $\tau_j(\beta)$. Thus the integral $I_j$ of \underline{\sf TraceIntegrals} is such that $I_j-\partial_x \sum B_s \ln F_s$ is reduced and has no residues, and thus is of first kind. Thus $I_j$ is of torsion.
\end{proof}

\subsection{Completeness}

Once trace integrals have been computed over $\mathbb{K}(\alpha)$ we can consider the conjugated sums
$$\tilde{I}_{i,j}=\sum\limits_{\alpha \in E^{-1}(0)} \alpha^j I_{i,\alpha},\quad j=0\dots \deg E-1$$
We now want to compute enough such integrals such that $I$ can be written as a linear combination of them and an integral of the first kind.

\begin{proof}[of Theorem \ref{thm1}]

Consider the trace over $\mathbb{K}(\alpha_1)\; $ where $\alpha_1$ is a residue of the integral, and so a root of $R$. One of the factor $R_i(\lambda^{k'})$ of $R$ is the minimal unitary polynomial of $\alpha_1$, and note  $l=\deg R_i-1$. The other roots of $R_i(\lambda^{k'})$ are noted $\xi^i \alpha_j$. By assumption, its Galois group is $\mathbb{Z}_{k'}^{l+1} \ltimes S_{l+1}$. Now the trace over $\mathbb{K}(\alpha_1)$ of the roots of $R_i(\lambda^{k'})$ are
$$(\xi^j \alpha_1)_{j=0\dots k-1}, (\xi^j t )_{j=0\dots k-1},\dots, (\xi^j t )_{j=0\dots k'-1}$$
and $\alpha_1+lt=u\in\mathbb{K}$ where $u$ is minus the second leading coefficient of $R_i(\lambda^{k'})$ (it is zero for $k'>1$).
Now applying the Galois group of $R_i(\lambda^{k'})$, we can permute the root $\alpha_1$ to any root $\alpha_i$, and the factorization of $R_i(\lambda^{k'})$ in $\mathbb{K}(\alpha_i)[\lambda]$ will have the same structure. We obtain then from algorithm \underline{\sf TraceIntegrals} integrals whose residues are any line of the matrix
$$M=\left(\begin{array}{cccc} \alpha_1 & (u-\alpha_1)/l & \dots &(u-\alpha_1)/l \\ & & \dots & \\ (u-\alpha_{l+1})/l & \dots & (u-\alpha_{l+1})/l & \alpha_{l+1} \end{array} \right)$$
and their multiples by $\xi$. This matrix is invertible if $u\neq 0$, and $\hbox{Im} M= \{x\in\mathbb{C}^{l+1}, \sum x_i=0\}$ for $u=0$. Thus $(\alpha_1,\dots,\alpha_{l+1})$ is in the image of $M$ in both cases, and thus a $\mathbb{K}$ linear combination of the integrals of \underline{\sf TraceIntegrals} will have the residues $(\alpha_1,\dots,\alpha_{l+1})$ at suitable poles. Doing this for all (conjugacy classes of) residues, we can subtract to $I$ a linear combination of integrals of \underline{\sf TraceIntegrals} removing all residues, and thus all poles, so leaving an integral of the first kind.
\end{proof}

Similar proofs can be done with smaller Galois group. In particular, the same proof works when replacing $S_{l+1}$ by any $2$ transitive group, and other groups could lead to a different matrix $M$, but still invertible.\\

\textbf{Example:} (see \cite{11}) $\mathcal{I}_1=$
$$\frac{535423}{(x^4-8x^3+236x^2-880x+12964)(x-15)(x^2+118)^{1/3}}$$
The residues are solutions up to multiplication by $\xi$ of 
$$\lambda-1, \lambda^2+\frac{3527}{220}\lambda\xi+\frac{11261}{5280}\lambda-\frac{449219897}{6082560}-\frac{12314729}{276480}\xi,$$
$$ \lambda^2+\frac{73387}{5280}\lambda\xi-\frac{11261}{5280}\lambda-\frac{12314729}{276480}-\frac{449219897}{6082560}\xi$$
Now applying \underline{\sf TraceIntegrals} with these extensions gives
\begin{small}\begin{equation}\label{exa}
\frac{174584x^4+700160x^3-45841128x^2+306988544x-11145996240}{(x^4-8x^3+236x^2-880x+12964)(x-15)(x^2+118)^{1/3}}
\end{equation}\end{small}
for the trace over $\mathbb{K}$ and $4$ more complicated expressions for the two extensions of degree $2$.
For the integrals
$$\int (x^n+x-3)^{-1}(x^2+118)^{-1/3}$$
we obtain for $n=2$ with $\alpha^6-\tfrac{1}{191867}\alpha^3-\tfrac{1}{32425523}=0$
$$\frac{13\alpha^2(2494271\alpha^3-29531)}{(2494271\alpha^3-243x-128)(x^2+118)^{1/3}}.$$
\begin{center}
\begin{tabular}{|c|c|c|c|c|c|}\hline
$n$      & 4   & 5   & 6     & 7      & 8 \\\hline
Degree   & 12  & 15  & 18    & 21     & 24 \\\hline
Galois   & 1944 &29160&524880&11022480& 264539520  \\\hline
Time     & 1.21&10.3 & 31.6  &1138    & 2333  \\\hline
\end{tabular}\\
\end{center}
Galois groups of the residue polynomials $R$ have been computed with Magma, but this computation is not necessary to perform the algorithm, however this ensures that $R$ is generic and show how useful it is to avoid computations in the splitting field. Remark that the trace integrals do not always split the poles of the integral when there are $\mathbb{K}$ relations between the residues, and in particular in $\mathcal{I}_1$ the $\mathbb{K}$-dimension of the residues is $2$ instead of expected $5$ (but this is still a generic $R$!).

\section{Computations in Jacobians}

\subsection{Superelliptic divisors}

Let us recall the definition of divisor and a introduce a specific notion for superelliptic curves.

\begin{defi}
A divisor $D$ on a curve $\mathcal{C}$ is a function $\mathcal{C} \rightarrow \mathbb{Z}$ with finite support. It is said to be principal if there exists a rational function $f$ on $\mathcal{C}$ such that $D(z)=\hbox{ord}_{x=z} f(x)$. It is said to be of torsion if there exists $N\in\mathbb{N}^*$ such that $ND$ is principal. The height of a divisor is $\sum_{(x,y)\in\mathcal{C}} \mid D(x,y) \mid$.\\
A superelliptic divisor $D$ on a superelliptic curve $\mathcal{C}$ is a function $\mathcal{C} \rightarrow \mathbb{Z}[\xi]$ with finite support and $D(\sigma(z))=\xi D(z)$ with $\sigma:\mathcal{C} \rightarrow \mathcal{C}$ the $k$th order shift on branches. It is said to be principal if $D=\sum_{i=1}^l a_i D_i$ with $a_i\in\mathbb{Z}[\xi]$ and $D_i$ principal divisors. It is said to be of torsion if there exists $N\in\mathbb{Z}[\xi]^*$ such that $ND$ is principal. The superelliptic divisor of a superelliptic integral is defined by
$$D(z)= \hbox{res}_z P/(QS^{1/k})$$
provided that all the residues are in $\mathbb{Z}[\xi]$.
\end{defi}

The divisors are usually defined as a function on the places of $\bar{\mathcal{C}}$, which can be different than simply points of $\bar{\mathcal{C}}$, however the technical condition implies that any ramification point is maximally ramified including infinity, and thus there is a unique place over a ramification point. The value of the divisor at infinity is recovered using the fact that the sum over all points $\in\bar{\mathcal{C}}$ should be zero. Remark that the notion of torsion order for divisors is well defined (the minimal $N$), however it is not always the case for superelliptic divisors. The set of possible $N$ forms an ideal of $\mathbb{Z}[\xi]$, and from $k=23$, this ideal is not always principal. In the following, we will not try to find the optimal one anyway.

A divisor will be represented by a list of triples of a irreducible polynomial $Q$ in $x$, a polynomial $R$, and a list of $k$ integers. The roots of $Q$ are the abscissas of the support of $D$, $R$ evaluated at $Q^{-1}(0)$ defines the ordinate of a point at such abscissa, and the list are the value of the divisor at this point and the other obtained by multiplication by $\xi$ of the ordinate.

\begin{prop}\label{prop1}
Any superelliptic divisor $D$ can be written uniquely
\begin{equation}\label{eq3}
kD(z)=\sum\limits_{i=0}^{k-1} \xi^i \tilde{D}(\sigma^i(z))
\end{equation}
where $\tilde{D}$ is a divisor on $\mathcal{C}$ such that
\begin{equation}\label{eq4}
\sum\limits_{i=0}^{k-1} \tilde{D}(\sigma^i(z)) \xi^{ij}=0,\; \forall j \wedge k \neq 1
\end{equation}
and $\tilde{D}(z)=0$ on ramification points. We have $D$ of torsion if and only if $\tilde{D}$ is of torsion. Algorithm \underline{\sf Divisor} computes the $\tilde{D}$ associated to the superelliptic divisor of a superelliptic integral.
\end{prop}

\begin{proof}
Let us note $[d_0,\dots,d_{k-1}]$ the values of $\tilde{D}$ over a given abscissa (not ramified), and note $U(z)=\sum_{i=0}^{k-1} d_i z^i$. We have from \eqref{eq3}
$$D(z)=\sum\limits_{i=0}^{k-1} \xi^i \tilde{D}(\sigma^i(z))= \left[ \sum\limits_{i=0}^{k-1} \xi^i d_{i+l} \right]_{l=0\dots k-1}\!\!\!\!\!\!\!\!\!\!\!\!\!\!\!\!\!\!\!\!=\left[ \xi^{-l} U(\xi)\right]_{l=0\dots k-1} $$
where indices are taken modulo $k$. As shifting branches on $D$ multiplies it by $\xi$, this equality is satisfied if and only if it is satisfied for $l=0$.
We also know that $d_i\in\mathbb{Z}$, and thus we can apply on \eqref{eq3} the Galois action $\psi_j\in \hbox{Gal}(\mathbb{K}:\mathbb{Q})$ which substitutes $\psi_j(\xi)=\xi^j$ with $j\wedge k=1$. Thus we know the values of $U(\xi^j),\; j\wedge k=1$. The condition \eqref{eq4} is $U(\xi^j)=0,\; \forall j \wedge k \neq 1$. Thus we know $U$ on all roots of unity, and $U$ is of degree $\leq k-1$. By polynomial interpolation, there exists a unique $d$ satisfying these conditions. On ramification points, we have $D(z)=0$ thus \eqref{eq3} is satisfied with $\tilde{D}(z)=0$.

If $\tilde{D}$ is of torsion, then $\tilde{D}(\sigma^i(z)),\; i=0\dots k-1$ is also, and thus $D$ is of torsion. If $D$ is of torsion, the Galois action $\psi_j$ gives that $\psi_j(D)$ is also a torsion divisor. Then
$$\sum\limits_{j\wedge k=1} k\psi_j(D)= \sum\limits_{j\wedge k=1} \sum\limits_{i=0}^{k-1} \xi^{ij} \tilde{D}(\sigma^i(z))=$$
and using \eqref{eq4}
$$\sum\limits_{j=0}^{k-1} \sum\limits_{i=0}^{k-1} \xi^{ij} \tilde{D}(\sigma^i(z))= \sum\limits_{i=0}^{k-1} \left(\sum\limits_{j=0}^{k-1}\xi^{ij}\right) \tilde{D}(\sigma^i(z))=k\tilde{D}(z).$$
\end{proof}

\begin{prop}
The principality of a divisor on $\mathcal{C}$ does not depend of its values on ramification points.
\end{prop}

\begin{proof}
Consider a divisor $D'$ whose support is only on ramification points. Noting $x_i$ the abscissa of these points and $d_i$ the values of $D'$, the rational function on $\mathcal{C}$
$$\prod\limits_{i=1}^{\sharp S^{-1}(0)} (x-x_i)^{d_i}$$
has for divisor $D'$, and thus $D'$ is principal. Thus for a divisor $D$ on $\mathcal{C}$, we have
$$D \hbox{ principal} \Leftrightarrow D+D' \hbox{ principal}$$
and thus the principality of $D$ is independent of its values on the $z_i$.
\end{proof}

From now, we will thus work with divisors modulo the divisors over ramification points, and thus in their representation we will not consider values over ramification points.

\subsection{Reduction in the Jacobian}

\begin{prop}\label{prop1b}
A divisor $D$ on $\mathcal{C}$ with non negative values is principal if and only if there exists $f\in\mathbb{K}[y]_{<k}[x]$ such that
$$\hbox{ord}_{(x,y)} f= D(x,y),\forall (x,y)\in\mathcal{C}$$
and $\hbox{wdeg} f =\sum_{z\in \mathcal{C}} D(z)$ where $\hbox{wdeg}(x^iy^j)=ki+(\deg S)j$.
\end{prop}

\begin{proof}
The existence of a rational $f$ satisfying the order condition is equivalent to the condition of principality of a divisor after multiplying $f$ by a rational function in $x$ (which shifts all the values of $D$ over a given abscissa). As the quantity $D(x,y)$ is always non negative and $S$ has only simple roots, $f$ should then be a polynomial. The number of zeros of $f$ on $\mathcal{C}$ counting multiplicity is $\hbox{wdeg} f$. The number of zeros with multiplicity required by the order condition is $\sum_{(x,y)\in \mathcal{C}} D(x,y)$.
\end{proof}
Remark that the simple roots condition on $S$ is necessary, as for $\mathcal{C}:y^3-x^2(x^2+1)$, $y^2/x$ has not a pole at $0$, its divisor is always non negative, but is not polynomial.

If the divisor $D$ corresponds to a superelliptic divisor, with Proposition \ref{prop1}, we can recover the principality of the superelliptic divisor as it is the divisor of the function $L_S(f)$, using the fact that the function $L_S$ is invariant by multiplication of $f$ by a function of $x$ only.

With this proposition, testing principality of a divisor reduces to a linear system solving problem. However, the size of this system grows as the height of $D$, and as the coefficients of the divisor come from residues of the integral, the height of $D$ can be very large, rendering this approach impractical except for small examples.

Let us introduce a divisor reduction process. Recall that the Jacobian of $\mathcal{C}$ is defined by its divisors modulo the principal divisors. It is a $g$ dimensional Abelian variety, and thus it is possible to reduce divisors to a set of divisors depending of $g$ parameters.

\begin{prop}\label{prop2}
Given a divisor $D$ over $\mathcal{C}$ with non negative values, there always exists a principal divisor $D'$ such that $D'-D$ has at most $(k-1)(\deg S -1)/2$ points in its support.
\end{prop}

\begin{proof}
Consider a the expression
$$f_N=\sum\limits_{j=0}^{k-1} \left(\sum\limits_{i=0}^{\lfloor ( N+(k-1)(\deg S -1)/2 )/k-j\deg S/k \rfloor} \!\!\!\!\!\!\!\!\! a_{i,j} x^iy^j\right) $$
Its number of roots on $\mathcal{C}$ counting multiplicity is $N+(k-1)(\deg S -1)/2$, and using the technical condition, it has
$$\sum\limits_{j=0}^{k-1} \lfloor (N+(k-1)(\deg S -1)/2)/k-j\deg S/k +1 \rfloor=$$
$$k+ N+(k-1)(\deg S -1)/2 -(k-1)(\deg S +1)/2=N+1$$
parameters. We write down the condition
$${{f_N}_{\mid \mathcal{C}}}^{(j)}(x,y)=0,\; \forall j <D(x,y),\; \forall (x,y)\in\mathcal{C}$$
This is a linear system on the $a_{i,j}$ with $\sum_{(x,y)\in\mathcal{C}} D(x,y)$ equations. Thus for $N=\sum_{(x,y)\in\mathcal{C}} D(x,y)$, it always admits a non zero solution. Among the roots of $f_N$ there will be the points in the support of $D$ with required multiplicity, but also $(k-1)(\deg S -1)/2$ additional points (or multiplicity increases). Thus the divisor of $f_N$ minus $D$ will have at most $(k-1)(\deg S -1)/2$ points in its support.
\end{proof}

Remark that if $(k-1)(\deg S -1)/2=0$, then $k=1$ or $\deg S=1$. Proposition \ref{prop2} allows then to reduce $D$ to a divisor with empty support, so $0$, and thus all divisors are principal, which is indeed the case as the genus of $\mathcal{C}$ is then $0$. Also, the genus using Riemann Hurwitz formula is $g=(k-1)(\sharp S^{-1}(0)-1)/2$, and as $S$ has only simple poles, our reduction is optimal.

\subsection{Negabinary expansion}

Recall that any integer $n\in\mathbb{Z}$ can be written uniquely
$$n=\sum\limits_{i=0}^l a_i (-2)^i,\; a_i\in\{0,1\},\; l\in\mathbb{N}$$
Similarly, a divisor $D$ on $\mathcal{C}$ is defined by a vector of integers over each abscissa, and thus we can write
$$D=\sum\limits_{i=0}^l (-2)^i D_i,\; \hbox{Im}(D_i)\subset \{0,1\},\; l\in\mathbb{N}$$

\noindent\underline{\sf JacobianReduce}\\
\textsf{Input:} A divisor $D$ on $\mathcal{C}$.\\
\textsf{Output:} A sequence of polynomial $\in\mathbb{K}[y]_{<k}[x]$, and a reduced divisor $\bar{D}$
\begin{enumerate}
\item if $D=0$, return $[1],0$.
\item Reduce $D$ modulo $2$, note $D_0$ the rest and $D_1$ such that $D=D_0-2D_1$.
\item Compute $\tilde{f},\tilde{D}:=D_0+2$\underline{\sf JacobianReduce}($D_1$).
\item $N:=\sum_{(x,y)\in\mathcal{C}} \tilde{D}(x,y)$
\item Note $f=\sum_{j=0}^{k-1} Q_j(x) y^j$ with $\deg Q_j=\lfloor ( N+(k-1)(\deg S -1)/2 )/k-j\deg S/k \rfloor$
\item Solve the linear system $\hbox{ord}_{(x,y)} f \geq  \tilde{D}(x,y),\;\forall (x,y)\in\mathcal{C}$, and note $f$ the solution of lowest weighted degree.
\item Compute $D'$ the divisor of $f$. Return $[f,\tilde{f}],D'-\tilde{D}$.
\end{enumerate}

We will prove that algorithm \underline{\sf JacobianReduce} returns a zero reduced divisor if and only if the integral $I$ whose superelliptic divisor gave $D$ thanks to Proposition \ref{prop1} is of torsion, and then can be written up to a first kind integral
$$I= \sum\limits_{i=0}^{l} (-2)^i L_S(f_i).$$

\begin{proof}[of Theorem \ref{thm2}]
Let us prove by recurrence that $D+\hbox{\underline{\sf JacobianReduce}}(D)$ is a principal divisor. For $D$ of height $0$, step $1$ returns a correct answer. Now assume it is true for all divisor of height less than $D$.
In step $3$, there is a recursive call on $D_1$ which is the quotient of $D$ by $-2$. Thus $D_1$ has strictly smaller height than $D$, thus by hypothesis, $D_1+\hbox{\underline{\sf JacobianReduce}}(D_1)$ is a principal divisor. Thus $D-\tilde{D}=-2(\hbox{\underline{\sf JacobianReduce}}(D_1)+D_1)$ is a principal divisor. In step $7$, $D'$ is principal by construction, and thus
$$D+\hbox{\underline{\sf JacobianReduce}}(D)=D+D'-\tilde{D}$$
is principal.

Thus if $D$ is principal, $\hbox{\underline{\sf JacobianReduce}}(D)$ is also principal. Let us prove that if $D$ is principal, then \underline{\sf JacobianReduce}$(D)=0$. In step $7$, we have $D'\geq \tilde{D}$ by construction, and thus \underline{\sf JacobianReduce} always return non negative value divisors. So in step $3$, $\tilde{D}$ has non negative values, and is principal. Now applying Proposition \ref{prop1b}, we know that in steps $5,6$ we will find a $f$ such that $D'-\tilde{D}=0$. As $D'\geq \tilde{D}$, this will be reached for the minimal weighted degree solution of equation in step $6$, and this is the one we choose.

We must now check termination and complexity. Recursive calls in step $3$ are made on a divisor of strictly lower height, except $0$, which is dealt in step $1$. In step $6$, a solution $f$ always exists thanks to Proposition \ref{prop2} as $\tilde{D}$ has non negative values. For complexity, steps $6,7$ are in $N^\omega$ where $N$ is the height of $\tilde{D}$. However by Proposition \ref{prop2}, the outputted divisor of \underline{\sf JacobianReduce} is of height at most $(k-1)(\deg S -1)/2$, and $D_0$ is of height at most $k \deg Q$, and thus the cost is $O(((k-1)(\deg S -1)+k\deg Q)^\omega)$. In recursive calls, the support of $D_0$ is always at most $k \deg Q$, thus the same bound applies. The number of recursive calls is at most $\log_2 \hbox{height}(D)$, thus giving complexity $O((kd)^\omega \log_2 \hbox{height}(D)$. Now the coefficients of $D$ are from the residues of $I$, and thus the roots of the residue polynomial $R$, which comes from a resultant computation. Thus the the residues are bounded by a polynomial in the coefficients of $I$, and thus $\log_2 \hbox{height}(D)$ is in $O(h)$ with $h$ the height of the coefficients of $I$. Thus complexity is $O((kd)^\omega h)$.
\end{proof}

\textbf{Example} (see \cite{7})
$$\mathcal{I}_2=\frac{8(7\sqrt{5}-15)}{(x-20+8\sqrt{5})\sqrt{x^3+5x^2-40x+80}}$$
The integral $\mathcal{I}_2$ is a trace integral of a superelliptic integral over $\mathbb{K}(\sqrt{5})$. The divisor of $\mathcal{I}_2$ is $D_2=[x-20+8\sqrt{5}, -120+56\sqrt{5},[1,-1]]$.
Consider also the integral
$$\mathcal{I}_3= \frac{3}{(x-1)\sqrt{x^3+8}},\;\; D_3=[x-1,-3,[-1,1]]$$
We compute the divisor reduction in the Jacobian of $3^nD_2$ and $3^n D_3$
\begin{center}
\begin{tabular}{|c|c|c|c|c|c|}\hline
$n$                     & 3  & 4   & 5   & 6    & 7 \\\hline
$D_2$ time            &0.21& 0.06& 0.09&0.48  & 1.6  \\\hline
$\bar{D_2}$ digits     & 1  &  1  &  1  & 1    & 1 \\\hline
$D_3$ time            & 0.06&0.45& 20  &5109  &$>10^4$ \\\hline
$\bar{D_3}$ digits     & 828 &7446&67008&603071&$>10^6$\\\hline
$\!\!\! D_3$ mod time $\!\!\!\!\!$ &0.55 &0.62&0.86 &1.21  &1.23  \\\hline
\end{tabular}\\
\end{center}
The reduction time (in s) for $D_2$ is negligible, but $D_3$ reduction time grows exponentially instead of linearly. This is only because the coefficient size of the reduced divisor grows exponentially, and the timings become as expected when computing mod $65521$. The $D_2$ reductions do not grow in size, in fact because this is a torsion divisor and thus reductions are periodic in $n$.
Thus in practice, we want to run \underline{\sf JacobianReduce} either in positive characteristic, or on torsion divisors. For a good reduction prime, if the divisor is not principal mod $p$, it is not principal in characteristic $0$. For example \eqref{exa}, the divisor is
\begin{small}
$$[x^2+24\xi x+8x-48\xi-130, 8+22\xi+2x, [95909, -158035, 62126]],$$
$$[x^2-24\xi x-16x+48\xi-82, 2\xi x+8\xi+22, [62126, -158035, 95909]],$$
$$[x-15, -7-7\xi, [-10560, 21120, -10560]]$$
\end{small}
To test principality of this divisor, it is enough (!) to look for a $L_S$ function with a polynomial of degree $1928100$. Trager's approach reduces this to solving a linear system of this size. Applying \underline{\sf JacobianReduce} to it modulo $13$ allows to reduce this divisor in $0.39s$ to $[x+11,11,[1,0,0]]$ and thus it is not principal.

\section{Torsion of Divisors}

\subsection{Hasse Weil Bound}

\begin{defi}
A good reduction $(p,\mathcal{J})$ for a reduced superelliptic integral $I$ on $\mathcal{C}$ with coefficients in $\mathbb{K}(\alpha)$ is such that
\begin{itemize}
\item $p$ does not divide $k$.
\item $\mathcal{J}$ is a prime ideal factor in characteristic $p$ of $<\Phi_k(\xi),$ $P(\alpha)>$ where $P\in\mathbb{K}[z]$ is the minimal polynomial of $\alpha$.
\item All poles of $I$ and roots of $S$ stay distinct under reduction modulo $\mathcal{J}$.
\end{itemize}
It is a very good reduction when moreover $\mathcal{J}$ has a single point.
\end{defi}

Good reduction primes have important properties \cite{1,4}.
\begin{itemize}
\item All divisors of the Jacobian of a curve mod $p$ are of torsion.
\item The mod $p$ reduction on the Jacobian restricted to torsion divisors of order coprime with $p$ is an isomorphism.
\end{itemize}
Thus the reduction conserves the torsion order provided the divisor is of torsion with order coprime with $p$. Following \cite{1}, reducing with two different good primes allows to recover a unique candidate for the torsion order.

\begin{prop}[Hasse Weil bound]
The torsion order of a divisor on a curve of genus $g$ on $\mathbb{F}_{p^q}$ is less than $(1+\sqrt{p^q})^{2g}$.
\end{prop}

The torsion order modulo $p$ will be computed by testing principality of $nD$ for $n$ up to the bound $(1+\sqrt{p^q})^{2g}$. Thus we want to minimize the upper bound in which $q=\sharp \mathcal{J}^{-1}(0)$. Using Tchebotarev theorem \cite{9}, the probability of having a factor $\mathcal{J}$ with $\sharp \mathcal{J}^{-1}(0)=1$ is $(\phi(k)\deg P)^{-1}$, so we can increase probabilistically $p$ by a factor $\phi(k)\deg P$ to ensure $q=1$. As typically $pq < p^q$, we will then only consider very good reductions.

\subsection{Torsional Test}

\noindent\underline{\sf TorsionOrder}\\
\textsf{Input:} A divisor $D$ on $\mathcal{C}$, $\epsilon>0$.\\
\textsf{Output:} An integer $n$, candidate for torsion order.
\begin{enumerate}
\item Find $p_1<p_2<p_3$ primes, $p_3>1/\epsilon$, such that $p_i \nmid \Delta(Q\hbox{sf}(S))$ and $(\Phi_k(\xi),P(\alpha))$ has a prime ideal factor mod $p_i$ with one solution, and note them $\mathcal{J}_1,\mathcal{J}_2,\mathcal{J}_3$.
\item For $n\in\mathbb{N}^*$, Compute \underline{\sf JacobianReduce}($nD$) mod $\mathcal{J}_1$ until it reduces to zero.
\item For $m\in\mathbb{N}^*$, Compute \underline{\sf JacobianReduce}($mD$) mod $\mathcal{J}_2$ until it reduces to zero.
\item Solve equation $np_1^u=mp_2^v$, and if a solution, note $N=np_1^u$. Else return $0$.
\item Compute \underline{\sf JacobianReduce}($ND$) mod $\mathcal{J}_3$. If $0$, return $N$ else return $0$.
\end{enumerate}

\begin{proof}[of Theorem \ref{thm3}]
Steps $1$ compute three different very good reduction prime ideals with $p_3>1/\epsilon$. As modulo a good reduction prime all divisors are of torsion, then steps $2,3$ terminate. Now the true torsion order (if it exists) should be both of the forms $np_1^u$ and $mp_2^v$. As $p_1\neq p_2$, this equation has a most one solution. If none, then $D$ is not of torsion, thus algorithm returns $0$. Else we test in step $5$ if $ND$ is principal modulo $\mathcal{J}_3$. If $ND$ is not principal in characteristic $0$, its reduction modulo $\mathcal{J}_3$ is a random element in a group of at least $(1+\sqrt{p_3})^{2g}$ elements, thus its probability to be by chance $0$ is $\leq 1/(1+\sqrt{p_3})^{2g} \leq r^{-g}<\epsilon$. For $g=0$ all elements are principal and thus this case would not happen, and for $g\geq 1$, the probability is verified.

Now for complexity, in step $1$ we need to avoid prime factors of $\Delta(Q\hbox{sf}(S))$, which are in $O(\ln \Delta(Q\hbox{sf}(S))$. Being unlucky, it is possible that for the first primes, we either have a factor or that $(\Phi_k(\xi),P(\alpha))$ has no prime factor of degree $1$. The probability for factorization is $1/(\phi(k) \deg P)$, and thus we will find a prime in $O(\phi(k) \deg P\ln \Delta(Q\hbox{sf}(S))$ tests. For steps $2,3$, the Hasse Weil bound applies and we will find a suitable $N$ in less than $(1+\sqrt{p})^{2g}=O(p^g)=O((\phi(k) \deg P\ln \Delta(Q\hbox{sf}(S))^g)$ tests. Each of these tests cost $O((kd)^\omega h \ln n)$, and thus total cost is in $\tilde{O}( (kd)^\omega h (k \deg P\ln \Delta(Q\hbox{sf}(S))^g)$.

We note $r=\deg P$, and $d$ is the number of abscissa in the support of $D$, which is also bounded by the degree of $Q$. Finally $h$ is the logarithmic height of the coefficients, and thus $\ln \Delta(Q\hbox{sf}(S)=O(h)$. Thus the cost is $\tilde{O}( (kd)^{\omega+g} h^{g+1} r^g)$.
\end{proof}

Remark that step $5$ is important for checking with good probability that indeed the divisor is of torsion. In binary complexity, arithmetic in $\mathbb{F}_p$ costs $O(\ln p)$, and thus the checking cost will be in $O(\ln \epsilon)$. In arithmetic complexity, $\epsilon$ does not matter, in the examples $\epsilon=1$ was enough and did not left false positives.\\

\textbf{Example} (see \cite{6,5}) $\mathcal{I}_4=$
$$\frac{\tfrac{8}{29}(5x^3+267x^2+2688x-10240)(x^2+40x+512)^{-1}}{\sqrt{x^5+113x^4+4864x^3+102400x^2+1048576x+4194304}}$$
The divisor of this integral is $D_4=[x^2+40x+512, 8x+512, [-1,1]]$. In $2.3s$ \underline{\sf TorsionOrder} finds the candidate $29$ using primes $3,5$ and checking with $11$. \underline{\sf JacobianReduce}($29D_4$) reduces it to $0$ in $0.17s$, thus integral is of torsion, with candidate integral (not simplified!)
\begin{small}
$$\tfrac{1}{29}(L_S(x^2+40x+512)-2L_S(x^3+92x^2+2560x+4y+24576)+$$
$$4L_S(x^4+104x^3+4096x^2+73728x+524288)-$$
$$8L_S(-3x^3-248x^2-6144x-49152+(x+40)y)+$$
$$16L_S(x^3+78x^2+1792x-2y+12288)-$$
$$32L_S(x^5+106x^4+4688x^3+107520x^2+1277952x+6291456+$$
$$(2x^2+80x+1024)y)+64L_S(x^2+40x+512))$$
\end{small}
This proves that $\mathcal{I}_4$ is of torsion, and this expression is indeed an integral and so the integral of $\mathcal{I}_4$ is elementary.
Going back to integral $\mathcal{I}_1$, the divisor of \eqref{exa} has very good reduction for $p=13,19$, and is respectively of torsion order $2,19$ (found in $3.7s$). Thus no compatible torsion order is found in step $4$, and thus this is not a torsion integral, and thus the integral of $\mathcal{I}_1$ is not elementary.

\subsection{Elementary Integration Algorithm}

We can now put together all these parts. We first compute minimal polynomials for the residues modulo multiplication by $\xi$.\\

\noindent\underline{\sf Residues}\\
\textsf{Input:} A reduced superelliptic integral $I$.\\
\textsf{Output:} A list of irreducible polynomials in $\mathbb{K}[\lambda]$ whose solutions are the residues of $I$ up to multiplication by $\xi$.
\begin{enumerate}
\item Compute $R(\lambda)=\hbox{resultant}_x(P^k-\lambda^k Q'^k S,Q)$
\item Factorise $R=R_1\dots R_l$ in $\mathbb{K}[\lambda]$. $L=[]$.
\item For $i=1\dots l$, if $\forall j, R_i(\xi^j \lambda)\notin L$ then add $R_i$ to $L$.
\item Return $L$
\end{enumerate}

We will then run \underline{\sf TraceIntegrals} with the field extension generated by a root of a polynomial in $L$. Removing factors of $R$ having the same roots up to multiplication by $\xi$ avoid doing the same calculation several times.\\

\noindent\underline{\sf ElementaryIntegrate}\\
\textsf{Input:} A superelliptic integral $I$.\\
\textsf{Output:} An elementary expression or ``Not elementary'' or ``Not handled''
\begin{enumerate}
\item Apply \underline{\sf HermiteReduction}($I$). If FAIL, return ``Not elementary'', else note $\tilde{I}$ the reduced integral and $A$ the algebraic part.
\item $L=\hbox{\underline{\sf Residues}}(\tilde{I})$. For $i=1\dots \sharp L$ do
\begin{enumerate}
\item $T_i=\hbox{\underline{\sf Traceintegrals}}(\tilde{I},\mathbb{K}[L_i^{-1}(0)])$. For $j=1\dots \sharp T_i$ do
\begin{enumerate}
\item $D=\hbox{\underline{\sf Divisor}}(T_{i,j}),N=\hbox{\underline{\sf TorsionOrder}}(D,1)$, if $N=0$ return ``Not elementary''.
\item $(D',G_{i,j})=\hbox{\underline{\sf Jacobianreduce}}(D)$. If $d'\neq 0$ return ``Not elementary''.
\end{enumerate}
\end{enumerate}
\item Compute $\hbox{Ints}=[x^iS^{-1/k},0]_{i=0\dots \lfloor \deg S/k \rfloor},$
$$\!\!\!\!\!\!\!\!\!\!\!\! \left[\sum_{\alpha \in L_i^{-1}(0)} \!\!\!\! \alpha^s T_{i,j}, \!\!\!\! \sum_{\alpha \in L_i^{-1}(0)} \!\!\!\! \alpha^s \sum (-2)^r L_S(G_{i,j})
\right]_{\underset{ j=1\dots \sharp T_i,i=1\dots \sharp L}{s=0\dots \deg L_i-1}}$$
\item Look for a linear combination of the first elements of $\hbox{Ints}$ which gives $\tilde{I}$. If none, return ``Not handled''.
\item Apply this same linear combination to the second elements of $\hbox{Ints}$, obtain an expression $Out$.
\item If $I-\partial_x (A+Out)=0$ return $A+Out$ else return ``Not elementary''.
\end{enumerate}

\begin{prop}
If \underline{\sf ElementaryIntegrate} returns ``Not elementary'', then $I$ is not elementary. If \underline{\sf ElementaryIntegrate} returns an expression, this is an elementary expression of $I$. If \underline{\sf ElementaryIntegrate} returns ``Not handled'', the residue polynomial $R$ is not generic.
\end{prop}

\begin{proof}
In step $1$, if \underline{\sf HermiteReduction}($I$) fails, then $I$ is not elementary. If $\tilde{I}$ is elementary, then all the $T_{i,j}$ are of torsion with proposition \ref{proptor}. In step $2(a)i$ if \underline{\sf TorsionOrder} returns $0$, then $T_{i,j}$ is not of torsion, thus $\tilde{I}$ is not elementary. The same for step $2(a)ii$. In step $3$, the integrals of first elements and the second elements differ by an integral of first kind. Thus $\tilde{I}$ and $Out$ differ by an integral of first kind. Thus $I$ and $A+Out$ differ by an integral of first kind. If $I-\partial_x (A+Out)\neq 0$, then this is a non zero integral of first kind and thus not elementary.
Step $6$ is the only case returning an expression, it is elementary by construction and the result is checked in step $6$. Finally the case ``Not handled'' can only occur if $\tilde{I}$ is not a linear combination of the $T_{i,j}$ and an integral of first kind, and by Theorem \ref{thm1} this does not occur when $R$ is generic.

For complexity, the first loop is executed at most $d$ times, and note $e_i$ the degree of extension for computing $T_i$, which costs $O(d^2e_i^2)$. The dominant cost of steps $2(a)$ is the torsion order calculation, which cost $\tilde{O}( (kd)^{\omega+g} h^{g+1} e_i^g)$. This test is done $e_i$ times, which give total cost of
$$\sum_{i=1}^{\sharp L}  \tilde{O}( (kd)^{\omega+g} h^{g+1} e_i^{g+1}+d^2e_i^2)$$
As function of $e_i$, torsion cost is dominant in positive genus, and as $\sum_i e_i \leq kd$ by convexity the maximum is reached when $\sharp L=1$ and $e_1=kd$, which gives $\tilde{O}( (kd)^{\omega+2g+1} h^{g+1})$. The steps $4,5,6$ are linear algebra in dimension $\sum_i e_i^2$ which is maximized when $\sharp L=1$ and $e_1=kd$. Thus the cost is in $O((kd)^{2\omega})$, which is less than the torsion part for positive genus as $\omega<3$.
\end{proof}

\section{Conclusion}

We proved that generically, a similar decomposition as done by Trager for rational integrals can be done for superelliptic integrals, and thus unusable large field extensions can be avoided. However it is unproven that it is always possible. In particular, we would like to prove that the traces of the roots of a polynomial over its rupture field are enough to find all $\mathbb{K}$ linear relations among the roots. The known algorithms have still factorial in degree complexity \cite{10}, even if generically factorization in the rupture field is enough to find the relations. We then use fast multiplication techniques in Jacobians to test fast for principality and torsion of divisors. However, the principality test in characteristic $0$ is still slow in binary complexity as the coefficient size of reduced divisor grows very fast. Still we do this computation only when we are reasonably sure that the divisor is of torsion. Over $\mathbb{Q}$ for elliptic curves, the Nagell-Lutz Theorem gives a bound on the size of torsion points. If we had a similar bound for the size of torsion points in the Jacobian of superelliptic curves on number fields, we could then ensure that the principality test would also be fast in binary complexity. Also, our algorithm for finding torsion order relies to test principality of lots of multiples of $D$. As the cost is logarithmic, the total cost is $\tilde{O}(N)$ for computing all of them up to $N$. However, all $N$ are probably not possible, as for example for elliptic curves up to quartic fields we already have a complete (small) list of torsion orders.

\bibliographystyle{abbrv}

\begin{small}
\bibliography{super}

\begin{thebibliography}{10}

\bibitem{3}
L.~Bertrand.
\newblock Computing a hyperelliptic integral using arithmetic in the jacobian
  of the curve.
\newblock {\em Applicable Algebra in Engineering, Communication and Computing},
  6(4-5):275--298, 1995.

\bibitem{2}
M.~Bronstein.
\newblock Symbolic integration tutorial.
\newblock Citeseer, 1998.

\bibitem{4}
H.~Cohen, G.~Frey, R.~Avanzi, C.~Doche, T.~Lange, K.~Nguyen, and
  F.~Vercauteren.
\newblock {\em Handbook of elliptic and hyperelliptic curve cryptography}.
\newblock CRC press, 2005.

\bibitem{10}
C.~Fieker and W.~A. De~Graaf.
\newblock Finding integral linear dependencies of algebraic numbers and
  algebraic lie algebras.
\newblock {\em LMS Journal of Computation and Mathematics}, 10:271--287, 2007.

\bibitem{8}
M.~Kauers.
\newblock Integration of algebraic functions: a simple heuristic for finding
  the logarithmic part.
\newblock In {\em Proceedings of the twenty-first international symposium on
  Symbolic and algebraic computation}, pages 133--140, 2008.

\bibitem{11}
Y.~Kitaoka.
\newblock Notes on the distribution of roots modulo a prime of a polynomial.
\newblock {\em Uniform distribution theory}, 12(2):91--117, 2017.

\bibitem{9}
H.~Lenstra.
\newblock The chebotarev density theorem.
\newblock {\em URL: http://math. berkeley.
  edu/jvoight/notes/oberwolfach/Lenstra-Chebotarev. pdf}, 2006.

\bibitem{6}
F.~Lepr{\'e}vost.
\newblock Jacobiennes de certaines courbes de genre $2 $: torsion et
  simplicit{\'e}.
\newblock {\em Journal de th{\'e}orie des nombres de Bordeaux}, 7(1):283--306,
  1995.

\bibitem{7}
M.~A. Reichert.
\newblock Explicit determination of nontrivial torsion structures of elliptic
  curves over quadratic number fields.
\newblock {\em mathematics of computation}, 46(174):637--658, 1986.

\bibitem{5}
D.~SCHULTZ.
\newblock Trager?s algorithm for integration of algebraic functions revisited.

\bibitem{1}
B.~M. Trager.
\newblock {\em Integration of algebraic functions}.
\newblock PhD thesis, Massachusetts Institute of Technology, 1984.

\end{thebibliography}
\end{small}

\end{document}